\documentclass[psamsfonts]{amsart}  
\usepackage{amssymb}

\usepackage[dvips]{graphicx}

\usepackage{graphicx}
\usepackage{amssymb}                       
\usepackage{amsmath}
\usepackage{color}
\usepackage{times}

\usepackage[unicode,bookmarks,colorlinks]{hyperref}
\hypersetup{
   linkcolor=brickred,
}

\definecolor{mahogany}{cmyk}{0, 0.77, 0.87, 0}
\definecolor{salmon}{cmyk}{0, 0.53, 0.38, 0}
\definecolor{melon}{cmyk}{0, 0.46, 0.50, 0}
\definecolor{yellowgreen}{cmyk}{0.44, 0, 0.74, 0}
\definecolor{brickred}{cmyk}{0, 0.89, 0.94, 0.28}
\definecolor{OliveGreen}{cmyk}{0.64, 0, 0.95, 0.40}
\definecolor{RawSienna}{cmyk}{0, 0.72, 1.0, 0.45}
\definecolor{ZurichRed}{rgb}{1, 0, 0} 

\usepackage{fancyhdr}
\usepackage{amsmath,amstext,amssymb,amsopn,amsthm}
\usepackage{amsmath,amssymb,amsthm}
\usepackage[mathscr]{eucal}

\begin{document}

\newtheorem{lemma}[thm]{Lemma}
\newtheorem{proposition}{Proposition}
\newtheorem{theorem}{Theorem}[section]
\newtheorem{deff}[thm]{Definition}
\newtheorem{case}[thm]{Case}
\newtheorem{prop}[thm]{Proposition}
\newtheorem{example}{Example}

\newtheorem{corollary}{Corollary}

\theoremstyle{definition}
\newtheorem{remark}{Remark}

\numberwithin{equation}{section}
\numberwithin{definition}{section}
\numberwithin{corollary}{section}

\numberwithin{theorem}{section}

\numberwithin{remark}{section}
\numberwithin{example}{section}
\numberwithin{proposition}{section}

\newcommand{\gap}{\lambda_{2,D}^V-\lambda_{1,D}^V}
\newcommand{\gapR}{\lambda_{2,R}-\lambda_{1,R}}
\newcommand{\bD}{\mathrm{I\! D\!}}
\newcommand{\calD}{\mathcal{D}}
\newcommand{\calA}{\mathcal{A}}

\newcommand{\conjugate}[1]{\overline{#1}}
\newcommand{\abs}[1]{\left| #1 \right|}
\newcommand{\cl}[1]{\overline{#1}}
\newcommand{\expr}[1]{\left( #1 \right)}
\newcommand{\set}[1]{\left\{ #1 \right\}}

\newcommand{\calC}{\mathcal{C}}
\newcommand{\calE}{\mathcal{E}}
\newcommand{\calF}{\mathcal{F}}
\newcommand{\Rd}{\mathbb{R}^d}
\newcommand{\BR}{\mathcal{B}(\Rd)}
\newcommand{\R}{\mathbb{R}}
\newcommand{\al}{\alpha}
\newcommand{\RR}[1]{\mathbb{#1}}
\newcommand{\bR}{\mathrm{I\! R\!}}
\newcommand{\ga}{\gamma}
\newcommand{\om}{\omega}
\newcommand{\A}{\mathbb{A}}
\newcommand{\bH}{\mathbb{H}}

\newcommand{\bb}[1]{\mathbb{#1}}
\newcommand{\bI}{\bb{I}}
\newcommand{\bN}{\bb{N}}

\newcommand{\uS}{\mathbb{S}}
\newcommand{\M}{{\mathcal{M}}}
\newcommand{\calB}{{\mathcal{B}}}

\newcommand{\W}{{\mathcal{W}}}

\newcommand{\m}{{\mathcal{m}}}

\newcommand {\mac}[1] { \mathbb{#1} }

\newcommand{\bC}{\Bbb C}

\newtheorem{rem}[theorem]{Remark}
\newtheorem{dfn}[theorem]{Definition}
\theoremstyle{definition}
\newtheorem{ex}[theorem]{Example}
\numberwithin{equation}{section}

\newcommand{\Pro}{\mathbb{P}}
\newcommand\F{\mathcal{F}}
\newcommand\E{\mathbb{E}}
\newcommand\e{\varepsilon}
\def\H{\mathcal{H}}
\def\t{\tau}

\newcommand{\blankbox}[2]{%
  \parbox{\columnwidth}{\centering
    \setlength{\fboxsep}{0pt}%
    \fbox{\raisebox{0pt}[#2]{\hspace{#1}}}%
  }%
}
\title[Bellman function]{On the Bellman function of Nazarov, Treil and Volberg}

\author{Rodrigo Ba\~nuelos}\thanks{R. Ba\~nuelos is supported in part  by NSF Grant
\# 0603701-DMS}
\address{Department of Mathematics, Purdue University, West Lafayette, IN 47907, USA}
\email{banuelos@math.purdue.edu}
\author{Adam Os\c ekowski}\thanks{A. Os\c ekowski is supported in part by the NCN grant DEC-2012/05/B/ST1/00412.}
\address{Department of Mathematics, Informatics and Mechanics, University of Warsaw, Banacha 2, 02-097 Warsaw, Poland}
\email{ados@mimuw.edu.pl}

\begin{abstract}
We give an explicit formula for the Bellman function associated with the dual bound related to the unconditional constant of the Haar system.
\end{abstract}

\maketitle

\section{Introduction}
Let $\mathfrak{h}=(\mathfrak{h}_n)_{n\geq 0}$ denote the standard Haar system on $[0,1)$. Recall that this family of functions is given by 
\begin{align*}
&\mathfrak{h}_0=[0,1), && \mathfrak{h}_1=[0,1/2)-[1/2,1),\\
&\mathfrak{h}_2=[0,1/4)-[1/4,1/2), && \mathfrak{h}_3=[1/2,3/4)-[3/4,1),\\
&\mathfrak{h}_4=[0,1/8)-[1/8,1/4),&& \mathfrak{h}_5=[1/4,3/8)-[3/8,1/2),\\
&\mathfrak{h}_6=[1/2,5/8)-[5/8,3/4),&& \mathfrak{h}_7=[3/4,7/8)-[7/8,1), \,\,\ldots
\end{align*}
where we have identified a set with its indicator function. A classical result of Schauder \cite{Sch} states that the Haar system forms a basis of $L^p=L^p(0,1)$, $1\leq p<\infty$ (with the underlying measure being the Lebesgue measure). That is, for every $f\in L^p$ there is a unique sequence $a=(a_n)_{n\geq 0}$ of real numbers satisfying $||f-\sum_{k=0}^n a_k\mathfrak{h}_k||_p\to 0$. Let $\beta_p(\mathfrak{h})$ be the unconditional constant of $\mathfrak{h}$, i.e. the least extended real number $\beta$ with the following property: if $n$ is a nonnegative integer and $a_0,\,a_1,\,\ldots,\,a_n$ are real numbers such that $||\sum_{k=0}^n a_k \mathfrak{h}_k||_p\leq 1$, then
\begin{equation}\label{pp}
 \left|\left|\sum_{k=0}^n \e_k a_k \mathfrak{h}_k\right|\right|_p\leq \beta 
\end{equation}
for all choices of signs $\e_0,\,\e_1,\,\ldots,\,\e_n$. Using Paley's inequality \cite{P}, Marcinkiewicz \cite{M} proved that $\beta_p(\mathfrak{h})<\infty$ if and only if $1<p<\infty$. The precise value of $\beta_p(\mathfrak{h})$ was determined by Burkholder in \cite{Bu1}: we have
\begin{equation}\label{bet}
 \beta_p(\mathfrak{h})=p^*-1,\qquad 1<p<\infty,
\end{equation}
where $p^*=\max\{p,p/(p-1)\}$. Actually, the constant remains the same if we allow the coefficients $a_0$, $a_1$, $a_2$, $\ldots$ to take values in a Hilbert space $\mathcal{H}$ (cf. \cite{Bu2}). This result can be further generalized: if $(a_n)_{n\geq 0}$, $(b_n)_{n\geq 0}$ are sequences with $\mathcal{H}$-valued terms satisfying $|a_n|\leq |b_n|$ for each $n$, then
\begin{equation}\label{inh}
 \left|\left|\sum_{k=0}^n a_k\mathfrak{h}_k\right|\right|_p\leq (p^*-1)\left|\left|\sum_{k=0}^n b_k\mathfrak{h}_k\right|\right|_p,\qquad n=0,\,1,\,2,\,\ldots,\quad 1<p<\infty,
\end{equation}
and the constant $p^*-1$ cannot be replaced by a smaller one. 
 The original proof of this fact exploits the properties of a certain special functions, the associated Bellman function (for details, see Burkholder \cite{Bu1, Bu2, Bu3}).  Burkholder's sharp martingale inequalities have been widely sued to obtain tight bounds for a large class of operators, including many classical Fourier multipliers.  See the \cite{BanOse} and the many references contained therein.  

In the nineties, Nazarov, Treil and Volberg (cf. \cite{NT} and a preprint version of \cite{NTV}) proposed a different, dual approach to the above $p^*-1$ problems. Namely, they proved that \eqref{bet}, \eqref{inh} can be deduced from the existence of a function $B_p$ defined on the set 
$$\mathcal{D}=\big\{(\zeta,\eta,Z,H)\in\mathcal{H}\times \mathcal{H}\times [0,\infty)\times [0,\infty):Z\geq |\zeta|^p,\,H\geq |\eta|^q\big\},$$ 
satisfying the following two conditions:
\begin{itemize}
\item[(I)] We have $0\leq B_p(\zeta,\eta,Z,H)\leq (p^*-1)Z^{1/p}H^{1/q}$ on $\mathcal{D}$.
\item[(II)] For any $a_\pm=(\zeta_\pm,\eta_\pm,Z_\pm,H_\pm)\in \mathcal{D}$, we have the concavity-type condition
$$ B_p\left(\frac{a_-+a_+}{2}\right)-\frac{B_p(a_-)+B_p(a_+)}{2}\geq \left|\frac{\zeta_+-\zeta_-}{2}\right|\left|\frac{\eta_+-\eta_-}{2}\right|.$$
\end{itemize}
The existence of such a function can be extracted from Burkholder's works \cite{Bu1} and \cite{Bu2} via a dual formulation. As shown later by Nazarov and Volberg \cite{NV} and Dragi\v{c}evi\'c and Volberg \cite{DV}, this special object can be further exploited to yield interesting tight $L^p$ bounds for Riesz transforms in the classical setting and in the setting of the Ornstein--Uhlenbeck semigroup. 

Finding explicit formulas for Bellman functions is in general a rather nontrivial matter and there is an intriguing question about an explicit formula for ${B}_p$.  What is even more surprising is this case is that  this problem has been solved thus  far only in the particular case $p=2$ where the explicit expression is very easy to obtain. Indeed, for this value of the parameter $p$, Nazarov, Treil and Volberg \cite{NT,NV} showed that
\begin{equation}\label{p=2}
 \mathbb{B}_2(\zeta,\eta,Z,H)=\sqrt{(Z-|\zeta|^2)(H-|\eta|^2)}
\end{equation}
works just fine. 
The paper \cite{NT} contains also some attempts to find $B_p$ explicitly for other values of $p$, but with no success. Nevertheless, the authors managed to construct, for each $1<p<\infty$, a function which satisfies (II) and a version of (I), in which $p^*-1$ is replaced by a slightly larger constant. 
The purpose of this paper is to fill this gap and give an explicit formula for $B_p$ satisfying (I) and (II), for all $1<p<\infty$.  While at this point we do not have an application that takes advantage of this explicit expression, we believe such applications do exist.  For example, the upper bound estimate $1.575(p^*-1)$ for the norm of the Beurling-Ahlfors operator given by the first author and Janakiraman in \cite{BanJan}  arose from calculation with the explicit Bellman function discovered by Burkholder in the solution of the martingale transform problem.  While the bound $2(p^*-1)$ can be obtained from the existence of Burkholder's Bellman function, the better  $1.575(p^*-1)$ bound require the explicit expression.  This leads us to believe that, in the same way, the explicit expression for the Nazarov-Treil-Volberg Bellman function should lead to an improvement of the arguments in \cite{NV} which may yield a better bound.  

Suppose that $1<p\leq 2$ and introduce the function $\mathcal{B}_p:\mathcal{D}\to \R$ as follows.  If $|\eta|^qZ\geq |\zeta|^pH$, then
$$ \mathcal{B}_p(\zeta,\eta,Z,H)=\frac{(H-|\eta|^q)^{1/q}(Z-|\zeta|^p)^{1/p}}{p-1}.$$
On the other hand, if $|\eta|^qZ< |\zeta|^pH$, then
$$ \mathcal{B}_p(\zeta,\eta,Z,H)=\gamma Z^{1/p}H^{1/q}-|\zeta||\eta|Y,$$
where $(\gamma,Y)$, $0\leq Y<\gamma<(p-1)^{-1}$ is the unique solution to the system of equations
\begin{equation}\label{intr}
\frac{(1-(p-1)Y)(1+Y)^{p-1}}{(1-(p-1)\gamma)(1+\gamma)^{p-1}}=\frac{Z}{|\zeta|^p},\qquad \frac{Y(1+Y)^{p-2}}{\gamma(1+\gamma)^{p-2}}=\left(\frac{|\eta|^qZ}{|\zeta|^pH}\right)^{1/q}.
\end{equation}
The existence and  uniqueness of the pair $(\gamma,Y)$ will be shown later.  (See Lemma \eqref{auxilem} below.) 

Here is the precise statement of our main result. Throughout this paper,  $q=p/(p-1)$ denotes the conjugate exponent of $p$.

\begin{theorem}\label{mainthm}
For any $1<p\leq 2$, the function $\mathcal{B}_p$ satisfies (I) and (II). If $p>2$, then the function $(\zeta,\eta,Z,H)\mapsto \mathcal{B}_q(\eta,\zeta,H,Z)$ satisfies (I) and (II).
\end{theorem}
It is not difficult to check that when $p=2$, we get the function \eqref{p=2}.  In this case  the system \eqref{intr} can be solved explicitly, and in both cases $|\eta|^2Z\geq |\zeta|^2H$, $|\eta|^2Z<|\zeta|^2H$ we get the expression $\sqrt{(Z-|\zeta|^2)(H-|\eta|^2)}$. For other values of the parameter $p$, no similar compact formula for $\mathcal{B}_p$ seems to exist.

A few words about the proof of the above statement are in order. One can establish the theorem by the direct verification of the conditions (I) and (II), but this approach is extremely technical, and it does not give an  indication on how the special function is constructed. Thus, to simplify and clarify the reasoning, we  decided to propose a different proof. 
There is an abstract formula for a function satisfying the conditions (I) and (II) due to Nazarov and Treil \cite{NT} (see also Nazarov and Volberg \cite{NV} and Dragi\v{c}evi\'c and Volberg \cite{DV}).  We will derive the formula explicitly, actually with the use of a slightly more general, probabilistic setting. This approach has also the advantage that it shows how to handle complicated Bellman functions (depending on many variables) by solving associated less dimensional problems.  For more on this topic, see the second author's monograph \cite{Ose}. 

We have organized the remainder of this paper as follows. In the next section we present the abstract formula of Nazarov and Treil for the function satisfying (I) and (II) and express it in the probabilistic language of martingales. Section 3 contains some auxiliary material: we establish there a family of auxiliary $L^p$ estimates for martingales. The final two sections are devoted to the proof of our main result, Theorem \ref{mainthm}.

\section{An abstract formula}
 Let us start with introducing the necessary notation. Let $\mathfrak{D}$ denote the lattice of dyadic subintervals of $[0,1)$. Given $I\in \mathfrak{D}$, its left and right halves will be denoted by $I_-$ and $I_+$, respectively. Furthermore, for $I\in\mathfrak{D}$ and a locally integrable function $\varphi$ on $[0,1)$, we denote by $\varphi_I$ the average of $\varphi$ over $I$: $\varphi_I=\frac{1}{|I|}\int_I \varphi$. For a fixed $(\zeta,\eta,Z,H)\in\mathcal{D}$, consider all integrable $\varphi,\,\psi$ on $[0,1)$ which satisfy $\varphi_{[0,1)}=\zeta$, $\psi_{[0,1)}=\eta$, $(|\varphi|^p)_{[0,1)}\leq Z$ and $(|\psi|^q)_{[0,1)}\leq H$ (it is not difficult to see that such functions exist). Then, as shown by Nazarov and Treil \cite{NT}, the function
\begin{equation}\label{defB}
 \mathbb{B}_p(\zeta,\eta,Z,H)=\frac{1}{4}\sup\sum_{I\in \mathfrak{D}}|\varphi_{I_+}-\varphi_{I_-}||\psi_{I_-}-\psi_{I_+}||I|
 \end{equation}
satisfies (I) and (II). Here the supremum is taken over all $\varphi$, $\psi$ as above. We will show that the function of Theorem \ref{mainthm} coincides with $\mathbb{B}_p$. Observe that the roles of $\varphi$ and $\psi$ are symmetric, and therefore we immediately see that $\mathbb{B}_p(\zeta,\eta,Z,H)=\mathbb{B}_q(\eta,\zeta,H,Z)$ for all $(\zeta,\eta,Z,H)\in\mathcal{D}$. Consequently, we will be done with Theorem \ref{mainthm} if we manage to establish the equality $\mathbb{B}_p=\mathcal{B}_p$ for $1<p<2$.

Actually, it will be convenient for us to work with an appropriate probabilistic version of \eqref{defB}. Assume that $(\Omega,\F,\mathbb{P})$ is a probability space, equipped with the filtration $(\F_n)_{n\geq 0}$, a nondecreasing sequence of sub-$\sigma$-algebras of $\F$. Let $f,\,g$ be $\mathcal{H}$-valued martingales adapted to $(\F_n)_{n\geq 0}$, and denote by $(df_n)_{n\geq 0}$, $(dg_n)_{n\geq 0}$ the associated difference sequences:
$$ df_0=f_0,\qquad df_n=f_n-f_{n-1},\quad n=1,\,2,\,\ldots,$$
and similarly for $dg$. Following Burkholder \cite{Bu1}, we say that $g$ is \emph{differentially subordinate} to $f$, if for any $n\geq 0$ we have $|dg_n|\leq |df_n|$ almost surely. 

The triple $([0,1),\mathcal{B}([0,1)),|\cdot|)$ forms a probability space and $\mathfrak{D}$ gives rise to the corresponding dyadic filtration (for each $n$, the $\sigma$-algebra $\F_n$ is generated by the Haar functions $\mathfrak{h}_0$, $\mathfrak{h}_1$, $\ldots$, $\mathfrak{h}_n$). The adapted martingales in this special setting are called \emph{dyadic} (or Haar) martingales. 
We easily check that the formula \eqref{defB} can be rewritten as
$$ \mathbb{B}_p(\zeta,\eta,Z,H)=\sup \E \sum_{n=1}^\infty |df_n||dh_n|,$$
where the supremum is taken over the class of all dyadic martingales $f=(f_n)_{n\geq 0}$, $h=(h_n)_{n\geq 0}$ such that $f_0\equiv\zeta$, $\sup_n\E |f_n|^p\leq Z$, $h_0\equiv \eta$ and $\sup_n\E |h_n|^q\leq H$. Let us transform this formula to a more convenient form. First, note that we can write
$$ \mathbb{B}_p(\zeta,\eta,Z,H)=\sup \E \sum_{n=1}^\infty \langle dg_n,dh_n\rangle$$
($\langle\cdot,\cdot\rangle$ is the scalar product in $\mathcal{H}$), where the supremum is taken over all $f$, $h$ as above and all dyadic martingales $g$ which are differentially subordinate to $f$. This can be further simplified. Pick the martingales $f$, $g$, $h$ as above, and note that the first two of them are bounded in $L^p$, while the last one is bounded in $L^q$. Thus, using classical results from the martingale theory, there are random variables $f_\infty$, $g_\infty$ and $h_\infty$ such that $f_n\to f_\infty$, $g_n\to g_\infty$ in $L^p$ and $h_n\to h_\infty$ in $L^q$. Thus, by the orthogonality of the martingale differences, we get that 
\begin{equation}\label{defBB}
\begin{split}
 \mathbb{B}_p(\zeta,\eta,Z,H)&=\sup\E \left\langle \sum_{n=1}^\infty dg_n,\sum_{n=1}^\infty dh_n\right\rangle\\
&=\sup \E \big\langle g_\infty-g_0,h_\infty-h_0\big\rangle\\
&=\sup\big\{ \E \big\langle g_\infty, h_\infty\rangle-\langle \E g_\infty,\E h_\infty\rangle\big\},
\end{split}
\end{equation}
where the supremum is taken over all dyadic martingale triples $(f,g,h)$ such that $f_0\equiv \zeta$, $\E |f_\infty|^p\leq Z$, $h_0\equiv \eta$, $\E |h_\infty|^q\leq H$ and $g$ is differentially subordinate to $f$. This formula immediately shows that $\mathbb{B}_p(\zeta,\eta,Z,H)=\mathcal{B}_p(\zeta,\eta,Z,H)$ if $|\zeta|^p=Z$ or $|\eta|^q=H$; indeed, then the corresponding martingale ($f$ or $h$) must be constant and hence $\mathbb{B}_p(\zeta,\eta,Z,H)=0$. Thus, in our considerations below, we will assume that the strict estimates $|\zeta|^p<Z$ and $|\eta|^q<H$ hold true. Another crucial observation, particularly helpful during the study of lower bounds for $\mathbb{B}_p$, is the fact that in the above formula one can consider all (not necessarily dyadic) martingales. This follows from the results of Maurey \cite{Mau}, see also Section 10 in Burkholder's paper \cite{Bu1}.

The proof of Theorem \ref{mainthm} will rest on a careful analysis of the above formula for $\mathbb{B}_p$. It will consist of several ingredients which are presented in the three sections below.

\section{$L^p$ bounds for differentially subordinate martingales}\label{auxs}

We start with a family of certain auxiliary martingale inequalities. 
For fixed $1<p<2$ and $0<\gamma\leq (p-1)^{-1}$, we introduce the function $b_{p,\gamma}:\mathcal{H}\times \mathcal{H}\to \R$ by
$$ b_{p,\gamma}(x,y)=\begin{cases}
\displaystyle \left(\frac{\gamma}{\gamma+1}\right)^{p-2}(|x|+|y|)^{p-1}\left(|y|-\frac{|x|}{p-1}\right)& \mbox{if }|y|<\gamma |x|,\\
\displaystyle |y|^p-\frac{(2-p)\gamma^{p-1}+\gamma^{p-2}}{p-1}|x|^p & \mbox{if }|y|\geq \gamma |x|.
\end{cases}$$
One can easily verified, given the range of $p$,  that $b_{p,\gamma}$ is of class $C^1$ on $\mathcal{H}\times \mathcal{H}$.  We recall that the martingale $g$ is subordinate to the martingale $f$ if $\mathbb{P}(|dg_n|\leq |df_n|)=1$ for all $n\geq 1$.
We will establish the following statement.

\begin{theorem}\label{auxth}
Suppose that $f$, $g$ are $\mathcal{H}$-valued martingales such that $(f_0,g_0)\equiv (x,y)$ and such that $g$ is subordinate to $f$.  Then for any $p$ and $\gamma$ as above we have
\begin{equation}\label{aux_in}
\E |g_n|^p\leq \frac{(2-p)\gamma^{p-1}+\gamma^{p-2}}{p-1}\E |f_n|^p+b_{p,\gamma}(x,y),\qquad n=0,\,1,\,2,\,\ldots.
\end{equation}
\end{theorem}

To prove this theorem, we will require the following properties of $b_{p,\gamma}$.

\begin{lemma}\label{propb}
(i) There is an absolute constant $c_{p,\gamma}$, depending only on the parameters indicated, such that
$$ |b_{p,\gamma}(x,y)|\leq c_{p,\gamma}(|x|^p+|y|^p)$$
and
$$
 \left|\frac{\partial b_{p,\gamma}(x,y)}{\partial x}\right|+\left|\frac{\partial b_{p,\gamma}(x,y)}{\partial y}\right|\leq c_{p,\gamma}(|x|^{p-1}+|y|^{p-1}).
$$

(ii) For any $x$, $y\in \mathcal{H}$ we have the majorization
\begin{equation}\label{maj}
b_{p,\gamma}(x,y)\geq |y|^p-\frac{(2-p)\gamma^{p-1}+\gamma^{p-2}}{p-1}|x|^p.
\end{equation}

(iii) For any $x,\,y,\,h,\,k\in \mathcal{H}$ such that $|k|\leq |h|$, the function
$$
F_{x,y,h,k}(t)= b_{p,\gamma}(x+th,y+tk),\qquad t\in \R,
$$
is concave.
\end{lemma}
\begin{proof}
(i) This is straightforward: we leave the details to the reader.

(ii) Clearly we may assume that $\mathcal{H}=\R$ and $x,\,y\geq 0$. Furthermore, it suffices to show the majorization for $y<\gamma x$. Finally, by homogeneity, we may assume that $x+y=1$. Then the bound can be rewritten as
$$ \left(\frac{\gamma}{\gamma+1}\right)^{p-2}\left(1-\frac{px}{p-1}\right)- (1-x)^p+\frac{(2-p)\gamma^{p-1}+\gamma^{p-2}}{p-1}x^p\geq 0$$
for $x\geq (\gamma+1)^{-1}$. Denoting the left-hand side by $G(x)$, we compute that $G((\gamma+1)^{-1})=G'((\gamma+1)^{-1})=0$ and that (using $x+y=1$ and $y<\gamma x$) 
\begin{align*}
 G''(x)&=p(p-1)x^{p-2}\left[\frac{(2-p)\gamma^{p-1}+\gamma^{p-2}}{p-1}-\left(\frac{1-x}{x}\right)^{p-2}\right]\\
&\geq p(p-1)x^{p-2}\left[\frac{(2-p)\gamma^{p-1}+\gamma^{p-2}}{p-1}-\gamma^{p-2}\right]\\
&=p(2-p)(\gamma x)^{p-2}(1+\gamma)\geq 0.
\end{align*}
Thus, \eqref{maj} follows.

(iii) This property, as shown by Burkholder, is crucial in proving inequalities for differentially subordinate martingales. The function $F_{x,y,h,k}$ is of class $C^1$, so we will be done if we check that $F_{x,y,h,k}''(t)\leq 0$ for $t$ such that $0<|y+tk|<\gamma|x+th|$ or $0<|x+th|<|y+tk|/\gamma$. In the first case, we go back to Burkholder's calculation  (cf. page 17 in \cite{Bu3}): actually, the function
$$ t \mapsto (|x+th|+|y+tk|)^{p-1}\left(|y+tk|-\frac{|x+th|}{p-1}\right)$$
is concave on $\R$ for \emph{any} $x$, $y$, $h$, $k$ with $|k|\leq |h|$. To handle $F''_{x,y,h,k}(t)$ for $0<|x+th|<|y+tk|/\gamma$, note that we have the translation property 
$F_{x,y,h,k}(t+s)=F_{x+th,y+tk,h,k}(s)$ for all $t,\,s\in \R$, and hence it is enough to study the sign of the second derivative at $t=0$. We compute that
\begin{equation}\label{secder}
\begin{split}
 &\frac{\mbox{d}^2}{\mbox{d}t^2}\left[|y+tk|^p-\frac{(2-p)\gamma^{p-1}+\gamma^{p-2}}{p-1}|x+th|^p\right]\Bigg|_{t=0}\\
 &=p|y|^{p-2}|k|^2+p(p-2)|y|^{p-4}\langle y,k\rangle^2\\
&\qquad \qquad  -\frac{(2-p)\gamma^{p-1} +\gamma^{p-2}}{p-1}\left(p(p-2)|x|^{p-4}\langle x,h\rangle^2+p|x|^{p-2}|k|^2\right).
\end{split}
\end{equation}
Now, since $p$ is smaller than $2$, we immediately see that  $p|y|^{p-2}|k|^2\leq p(\gamma |x|)^{p-2}|k|^2$, $p(p-2)|y|^{p-4}\langle y,k\rangle^2\leq 0$ and
$$ p(p-2)|x|^{p-4}\langle x,h\rangle^2+p|x|^{p-2}|k|^2\geq p(p-1)|x|^{p-2}|h|^2.$$
Hence the second derivative \eqref{secder} is not larger than $ p(p-2)\gamma^{p-1}|x|^{p-2}|k|^2\leq 0$, and the claim follows.
\end{proof}

With this lemma, we now turn our attention to the main result of this section.

\begin{proof}[Proof of Theorem \ref{auxth}]
There is a well-known procedure established by Burkholder which enables the extraction of \eqref{aux_in} from the special function $b_{p,\gamma}$. Fix $f$, $g$, $n$ as in the statement. Of course we may and do assume that $\E |f_n|^p<\infty$, since otherwise the bound is trivial. Then $\E |f_k|^p<\infty$ for all $0\leq k\leq n$, and hence also $df_k,\,dg_k$ are $p$-integrable for these values of $k$. 
The key observation is that by Lemma \ref{propb} (iii) and the smoothness of $b_{p,\gamma}$, we have
\begin{align*}
 b_{p,\gamma}(f_{k+1},g_{k+1})
&= b_{p,\gamma}(f_k+df_{k+1},g_k+dg_{k+1})\\
 &\leq b_{p,\gamma}(f_k,g_k)+ \left\langle \frac{\partial b_{p,\gamma}(f_k,g_k)}{\partial x},df_{k+1}\right\rangle+ \left\langle \frac{\partial b_{p,\gamma}(f_k,g_k)}{\partial y},dg_{k+1}\right\rangle,
\end{align*}
for $k=0,\,1,\,2,\,\ldots,\,n-1$. Now by Lemma \ref{propb} (i) and the aforementioned $p$-integrability of the differences of $f$ and $g$, we see that both sides above are integrable. Taking expectation yields $\E b_{p,\gamma}(f_{k+1},g_{k+1})\leq \E b_{p,\gamma}(f_{k},g_{k})$ and hence, by \eqref{maj},
\begin{align*}
 \E \left[|g_n|^p-\frac{(2-p)\gamma^{p-1}+\gamma^{p-2}}{p-1}|f_n|^p\right]&\leq \E b_{p,\gamma}(f_n,g_n)\\
 &\leq \E b_{p,\gamma}(f_0,g_0)=b_{p,\gamma}(x,y).
 \end{align*}
This is precisely the assertion of the theorem.
\end{proof}

Let us conclude this section by making a simple observation which will be needed later. Namely, if the martingale $f$ in Theorem \ref{auxth} is assumed to be $L^p$ bounded, then so is $g$ (by Burkholder's inequality for differentially subordinate martingales) and we may let $n\to \infty$ in \eqref{aux_in} to obtain
\begin{equation}\label{aux_in2}
 \E |g_\infty|^p- \frac{(2-p)\gamma^{p-1}+\gamma^{p-2}}{p-1}\E |f_\infty|^p\leq b_{p,\gamma}(x,y).
\end{equation}

\section{Proof of $\mathbb{B}_p\leq \mathcal{B}_p$}

Our goal is now to deduce the above upper bound for $\mathbb{B}_p$ from Theorem \ref{auxth}. We start with three technical facts.

\begin{lemma}\label{auxilem}
Let $1<p<2$ and fix $(\zeta,\eta,Z,H)\in \mathcal{D}$ such that $Z>|\eta|^p$, $H>|\eta|^q$ and $|\eta|^qZ<|\zeta|^pH$. Then there is a unique pair $(\gamma,Y)$ satisfying the system \eqref{intr}.
\end{lemma}
\begin{proof}
For clarity purposes, we split the proof into several steps. 

\smallskip

\noindent \emph{Step 1. Auxiliary functions.} 
Consider $\kappa,\,\delta:[0,\infty)\to [0,\infty)$ given by $\kappa(t)=(1-(p-1)t)(1+t)^{p-1}$ and $\delta(t)=t(1+t)^{p-2}$. A direct differentiation shows that
$$ \kappa'(t)=-p(p-1)t(1+t)^{p-2}<0,\quad \delta'(t)=(1+t)^{p-3}(1+(p-1)t)>0$$
and
$$ \delta''(t)=(p-2)(1+t)^{p-4}(2+(p-1)t)<0.$$

\noindent \emph{Step 2. An easy case.} If $|\eta|=0$, the assertion of the lemma is clear as  the second equality in \eqref{intr} implies $Y=0$, and plugging this into the first equation gives
$ \kappa(\gamma)=|\zeta|^p/Z\in (0,1)$. But, as we have observed above, $\kappa$ is strictly decreasing and satisfies $\kappa(0)=1$, $\kappa((p-1)^{-1})=0$; thus the claim follows at once from the intermediate value property. Hence, from now on, we may assume that $\eta\neq 0$. 

\smallskip

\noindent \emph{Step 3. An extra function.} As we have shown above, $\delta$ is strictly increasing so for a given $Y> 0$ there is a unique $G(Y)>Y$ satisfying
$$ \delta(Y)=\left(\frac{|\eta|^qZ}{|\zeta|^pH}\right)^{1/q}\delta\big(G(Y)\big).$$
Of course, $G$ is a smooth function on $(0,\infty)$. Differentiating both sides above gives
$$ G'(Y)=\frac{\delta'(Y)}{\delta'(G(Y))}\left(\frac{|\zeta|^pH}{|\eta|^qZ}\right)^{1/q},$$
and hence $G'(Y)>1$. Indeed, $|\zeta|^pH/(|\eta|^qZ)>1$ by the assumption of the lemma, and $\delta'(Y)/\delta'(G(Y))>1$, because $G(Y)>Y$ and $\delta''<0$.

\smallskip

\noindent \emph{Step 4. Completion of the proof.} The assertion of the lemma will follow if we show that there is a unique $Y>0$ for which $G(Y)<(p-1)^{-1}$ and
$$ F(Y):=\kappa(Y)-\frac{Z}{|\zeta|^p}\kappa(G(Y))=0.$$
However, we have
$$ F'(Y)=\kappa'(Y)-\frac{Z}{|\zeta|^p}\kappa'(G(Y))G'(Y)>\kappa'(Y)-\frac{Z}{|\zeta|^p}\kappa'(G(Y)),$$
since $G'(Y)>1$ and $\kappa'(G(Y))<0$. Thus,
$$ F'(Y)=-p(p-1)Y(1+Y)^{p-2}\left[1-\frac{Z}{|\zeta|^p}\left(\frac{|\zeta|^pH}{|\eta|^qZ}\right)^{1/q}\right]>0$$
and it remains to note that $\lim_{Y\to 0}F(Y)=0$ (since $\gamma(Y)\to 0$ as $Y\to 0$), and $F(Y)$ is positive when $G(Y)$ approaches $(p-1)^{-1}$.
\end{proof}

\begin{lemma}\label{aux_lemma}
Fix nonzero $\zeta,\,\eta\in \mathcal{H}$ and two numbers $Z$, $H$ satisfying $Z> |\zeta|^p$ and $H> |\eta|^q$. 
Consider the function
\begin{align*}
 L(\gamma,Y)=-Y|\eta|+H^{1/q}&\bigg(\left(\frac{\gamma}{\gamma+1}\right)^{p-2}(1+Y)^{p-1}\left(Y-\frac{1}{p-1}\right)\\
&\qquad \qquad \qquad \qquad \qquad +\frac{(2-p)\gamma^{p-1}+\gamma^{p-2}}{p-1}\frac{Z}{|\zeta|^p}\bigg)^{1/p},
\end{align*}
defined for $0\leq Y\leq \gamma\leq (p-1)^{-1}$, and assume that $L$ attains its minimum at the point $(\gamma_0,Y_0)$.

(i) If $|\eta|^qZ\geq |\zeta|^pH$, then $\gamma_0=Y_0=(p-1)^{-1}$.

(ii) If $|\eta|^qZ< |\zeta|^pH$, then $(\gamma_0,Y_0)$ is the unique solution to the system \eqref{intr}.
\end{lemma}
\begin{proof}
Observe first that $L$ is continuous, so its minimum is attained and hence $(\gamma_0,Y_0)$ exists. A little computation shows that if $Y$ lies in the interval $ [0,(p-1)^{-1})$ and $\gamma\in (Y,(p-1)^{-1})$, then
\begin{align*}
 \frac{\partial L(\gamma,Y)}{\partial \gamma}&=(2-p)\gamma^{p-3}(1-(p-1)\gamma)\left[\frac{(1+Y)^{p-1}(1-(p-1)Y)}{(1+\gamma)^{p-1}(1-(p-1)\gamma)}-\frac{Z}{|\zeta|^p}\right]\\
 &=(2-p)\gamma^{p-3}(1-(p-1)\gamma)\left[\frac{\kappa(Y)}{\kappa(\gamma)}-\frac{Z}{|\zeta|^p}\right],
 \end{align*}
where $\kappa$ is the function introduced in the proof of Lemma \ref{auxilem}. This function is decreasing and vanishes at $(p-1)^{-1}$, so  for each $Y$ as above there is a unique $\gamma(Y)\in (Y,(p-1)^{-1})$ at which the partial derivative vanishes. Here the one-dimensional restriction $\gamma\mapsto L(\gamma,Y)$ attains its minimum. Therefore, we have one of three possibilities for the location of $(\gamma_0,Y_0)$. Namely, 
\begin{itemize}
\item[a)] $(\gamma_0,Y_0)=(\gamma(0),0)$,
\item[b)] $(\gamma_0,Y_0)=((p-1)^{-1},(p-1)^{-1})$
\end{itemize}
\begin{itemize}
\item[c)] $(\gamma_0,Y_0)$ lies in the triangle $\big\{(\gamma,Y):0<Y<\gamma<(p-1)^{-1}\big\}$.
\end{itemize}
The first possibility is easily  ruled out.  To see this, we compute that
\begin{align*}
&\frac{\partial L(\gamma,Y)}{\partial Y}=-|\eta|+H^{1/q}\left(\frac{\gamma}{\gamma+1}\right)^{p-2}(1+Y)^{p-2}Y\times\\
&\times\bigg(\left(\frac{\gamma}{\gamma+1}\right)^{p-2}(1+Y)^{p-1}\left(Y-\frac{1}{p-1}\right)+\frac{(2-p)\gamma^{p-1}+\gamma^{p-2}}{p-1}\frac{Z}{|\zeta|^p}\bigg)^{1/p-1},
\end{align*}
which becomes negative when $Y\to 0$. Thus, b) or c) holds true.

If $|\eta|^qZ<|\zeta|^pH$, we easily check that ${\partial L}/{\partial Y}$ is positive when $\gamma$, $Y$ are sufficiently close to $(p-1)^{-1}$ and hence b) is impossible. Therefore, c) must hold and $(\gamma_0,Y_0)$ satisfies
$$ \frac{\partial L(\gamma_0,Y_0)}{\partial \gamma}=\frac{\partial L(\gamma_0,Y_0)}{\partial Y}=0.$$
One easily verifies that this condition is precisely \eqref{intr}. 

It remains to consider the case $|\eta|^qZ\geq |\zeta|^pH$. Suppose that c) holds. Then $(\gamma_0,Y_0)$ would have to satisfy \eqref{intr}. But the first equality in this system would imply $\gamma_0>Y_0$ (by $Z/|\zeta|^p>1$ and the aforementioned monotonicity of $\kappa$), while the second equality would give $\gamma_0\leq Y_0$ (we have $|\eta|^pZ/(|\zeta|^qH)\geq 1$ and the function $\delta$ of Lemma \ref{auxilem} is increasing). The contradiction shows that b) must be true, and this completes the proof of the lemma.
\end{proof}

Finally, let us state a simple fact, the proof of which is left to the reader.

\begin{lemma}\label{monga}
The function 
$$\gamma\mapsto \frac{(2-p)\gamma^{p-1}+\gamma^{p-2}}{p-1}$$
is strictly decreasing on the interval $(0,(p-1)^{-1}]$ and its value at $(p-1)^{-1}$ equals $(p-1)^{-p}$.
\end{lemma}

We now proceed to the bound $\mathbb{B}_p\leq \mathcal{B}_p$. Fix $\zeta,\,\eta\in\mathcal{H}$ and $Z> |\zeta|^p$, $H> |\eta|^q$, and pick martingales $f$, $g$, $h$ as in the definition of $\mathbb{B}_p(\zeta,\eta,Z,H)$. Clearly, we may assume that $g_0\equiv 0$ as the formula does not depend on the starting variable of $g$. By H\"older inequality, we see that for any $\gamma\in (0,(p-1)^{-1}]$ and any $y\in \mathcal{H}$ such that $\langle y,\eta\rangle=|y||\eta|$, we have
\begin{align*}
\E \langle g_\infty,h_\infty\rangle& =- \langle y,\eta\rangle+\E \langle g_\infty+y,h_\infty\rangle\\
&\leq -|y||\eta|+\left(\E |g_\infty+y|^p\right)^{1/p}H^{1/q}\\
&\leq -|y||\eta| + \left(\E |g_\infty+y|^p-\frac{(2-p)\gamma^{p-1}+\gamma^{p-2}}{p-1}(\E |f_\infty|^p-Z)\right)^{1/p}H^{1/q}\\
&\leq -|y||\eta|+H^{1/q}\left(b_{p,\gamma}(\zeta,y)+\frac{(2-p)\gamma^{p-1}+\gamma^{p-2}}{p-1}Z\right)^{1/p},
\end{align*}
where in the last line we have used \eqref{aux_in2}. It will be convenient to write $b_{p,\gamma}^\R$ to indicate that we consider the function $b_{p,\gamma}$ defined on $\R\times \R$.
The above chain of inequalities, combined with \eqref{defBB}, implies that
\begin{equation}\label{infimum}
\begin{split}
& \mathbb{B}_p(\zeta,\eta,Z,H)\\
&\leq \inf\left\{-s|\eta|+H^{1/q}\left(b_{p,\gamma}^\R(|\zeta|,s)+\frac{(2-p)\gamma^{p-1}+\gamma^{p-2}}{p-1}Z\right)^{1/p}\right\},
\end{split}
\end{equation}
where the infimum is taken over all $\gamma\in (0,(p-1)^{-1}]$ and all $s\geq 0$. 
The remainder of this section is devoted to showing that this infimum is equal to $\mathcal{B}_p(\zeta,\eta,Z,H)$. For the sake of convenience and clarity, we again  split the reasoning into separate steps. 

\smallskip

\noindent \emph{Step 1. The case $\zeta=0$.} Then we have $b_{p,\gamma}^\R(|\zeta|,s)=s^p$. Furthermore,
\begin{equation*}\label{monot}
\frac{(2-p)\gamma^{p-1}+\gamma^{p-2}}{p-1}\geq (p-1)^{-p}\qquad \mbox{for }\gamma\in (0,(p-1)^{-1}],
\end{equation*}
by virtue of Lemma \ref{monga}. 
 Consequently, we see that the infimum in \eqref{infimum} equals
\begin{equation}\label{infim}
 \inf_{s\geq 0}\left(-s|\eta|+ H^{1/p}(s^p+(p-1)^{-p}Z)^{1/p}\right). 
\end{equation}
However, a straightforward analysis of the derivative shows that the expression in  parentheses attains its minimal value for $s$ satisfying $s^p=(p-1)^{-p}|\eta|^qZ/(H-|\eta|^q)$. Plugging this $s$ into the expression in \eqref{infim}, we get that this infimum equals
$$ \frac{Z^{1/p}(H-|\eta|^q)^{1/q}}{p-1}=\mathcal{B}_p(0,\eta,Z,H).$$

\smallskip

\noindent \emph{Step 2. The case $\zeta\neq 0$, $|\zeta|^pH\leq |\eta|^qZ$.} In this case the function $b_{p,\gamma}^\R$ is homogeneous of order $p$. Take $|\zeta|$ out from the expression on the right in \eqref{infimum} to get that 
\begin{equation}\label{infimum2}
\begin{split}
& \mathbb{B}_p(\zeta,\eta,Z,H)\\
&\leq |\zeta|\inf\left\{-Y|\eta|+H^{1/q}\left(b_{p,\gamma}^\R(1,Y)+\frac{(2-p)\gamma^{p-1}+\gamma^{p-2}}{p-1}\frac{Z}{|\zeta|^p}\right)^{1/p}\right\}\\
&=|\zeta|\inf w(\gamma,Y),
\end{split}
\end{equation}
where the infimum is taken over the set $\{(\gamma,Y):\gamma\in (0,(p-1)^{-1}],\,Y=s/|\zeta|\geq 0\}$. Let us analyze the function $w$ separately on the following subsets of this domain. 
\begin{align*}
S_1&=\{(\gamma,Y): 0\leq Y\leq \gamma\leq (p-1)^{-1},\,Y<\gamma^{-1}\},\\
S_2&=\big\{(\gamma,Y):\gamma<\min\{Y,(p-1)^{-1}\}\big\},\\
S_3&=\{(\gamma,Y):\gamma=(p-1)^{-1},\, Y\geq (p-1)^{-1}\}.
\end{align*}
First, note that the infimum in \eqref{infimum2} cannot be attained on  $S_1$. This follows from Lemma \ref{aux_lemma} (i), since for $Y\leq \gamma$ we have $w=L$. On the other hand, the infimum cannot be attained on $S_2$ either. Indeed, for $(\gamma,Y)\in S_2$ we have
$$ w(\gamma,Y)=-Y|\eta|+H^{1/q}\left(Y^p+\frac{(2-p)\gamma^{p-1}+\gamma^{p-2}}{p-1}\left(\frac{Z}{|\zeta|^p}-1\right)\right)^{1/p},$$
which is strictly decreasing with respect to $\gamma$ as seen from Lemma \ref{monga}. Therefore, we see that during the computation of the right-hand side of \eqref{infimum2}, we may assume that $\gamma=(p-1)^{-1}$ and $Y\geq (p-1)^{-1}$. This leads us to the problem of finding the minimal value of the function
\begin{equation}\label{defF}
 F(Y)=
-Y|\eta|+ H^{1/q}\left(Y^p+(p-1)^{-p}\left(Z/{|\zeta|^p}-1\right)\right)^{1/p}
\end{equation}
on $[(p-1)^{-1},\infty)$. A straightforward analysis shows that this function attains its minimum at 
\begin{equation}\label{specY}
 Y=(p-1)^{-1}\left(\frac{Z-|\zeta|^p}{H-|\eta|^q}\frac{|\eta|^q}{|\zeta|^p}\right)^{1/p}.
\end{equation}
We also note that this value of $Y$ is at least $(p-1)^{-1}$, by the assumption $|\zeta|^pH\leq |\eta|^qZ$. It suffices to note that the minimum is precisely 
$$\frac{(Z-|\zeta|^p)^{1/p}(H-|\eta|^q)^{1/q}}{p-1}=\mathcal{B}(\zeta,\eta,Z,H).$$ 

\smallskip

\noindent \emph{Step 3. The case $\zeta\neq 0$, $|\eta|^qZ<|\zeta|^pH$.} We proceed as previously and observe that \eqref{infimum2} holds true.  We now  analyze $w$ on the sets 
\begin{align*}
S_1&=\{(\gamma,Y): 0\leq Y\leq \gamma\leq (p-1)^{-1}\},\\
S_2&=\big\{(\gamma,Y):\gamma<\min\{Y,(p-1)^{-1}\}\big\},\\
S_3&=\{(\gamma,Y):\gamma=(p-1)^{-1},\, Y> (p-1)^{-1}\}, 
\end{align*}
 separately. On the first set, we make use of Lemma \ref{aux_lemma}. Then  we have $w=L$, so by part (ii) of that statement the infimum (at least over $S_1$) is attained at the point satisfying \eqref{intr}. The same analysis as above shows that the set $S_2$ does not contribute to the infimum. Thus, all that remains is to check the behavior of $w$ on $S_3$ and this leads us to the function $F$ given by \eqref{defF}. However, this function is strictly increasing on $[(p-1)^{-1},\infty)$ (since $|\eta|^qZ<|\zeta|^pH$, the point $Y$ given by \eqref{specY} lies below $(p-1)^{-1}$) and hence the claim follows.

\section{Proof of $\mathbb{B}_p\geq \mathcal{B}_p$}

We now turn our attention to the proof of the lower bound for $\mathbb{B}_p$, which will show that the functions $\mathbb{B}_p$ and $\mathcal{B}_p$ actually coincide. As is usual, this will be accomplished by constructing appropriate examples. Fix a small $\delta>0$, numbers $\gamma\in (0,(p-1)^{-1})$, $Y\in [0,\gamma)$ and let $(\mathfrak{f},\mathfrak{g})$ be a Markov martingale with values in  $[0,\infty)\times \R$, satisfying the following conditions:
\begin{itemize}
\item[(i)] We have $(\mathfrak{f}_0,\mathfrak{g}_0)\equiv (1,Y)$.
\item[(ii)] A point of the form $(x,y)$ with $0<y<\gamma x$, leads to $\left(\frac{x+y}{\gamma+1},\frac{\gamma(x+y)}{\gamma+1}\right)$ or to $(x+y,0)$.
\item[(iii)] A point of the form $(x,y)$ with $-\gamma x<y<0$, leads to $\left(\frac{x-y}{\gamma+1},\frac{\gamma(-x+y)}{\gamma+1}\right)$ or to $(x-y,0)$.
\item[(iv)] A point of the form $(x,0)$ leads to $(x(1+\delta),\delta x)$, $(x(1+\delta),-\delta x)$, $\left(\frac{x}{\gamma+1},\frac{\gamma x}{\gamma+1}\right)$ or to $\left(\frac{x}{\gamma+1},-\frac{\gamma x}{\gamma+1}\right)$, with probabilities $\gamma/(2\gamma+2\delta(\gamma+1))$, $\gamma/(2\gamma+2\delta(\gamma+1))$, $\delta(\gamma+1)/(2\gamma+2\delta(\gamma+1))$ and $\delta(\gamma+1)/(2\gamma+2\delta(\gamma+1))$, respectively.
\item[(v)] All the points not mentioned in (ii) and (iii) are absorbing.
\end{itemize}
We need not specify the probabilities in (ii) and (iii).  These   are uniquely determined by the martingale property. To gain some intuition about this martingale pair, let us briefly describe its behavior for $Y>0$. The pair starts from $(1,Y)$ and it moves along the line of slope $-1$, either to the point on the line $y=\gamma x$, or to the $x$-axis. If the first possibility occurs, the pair stops and if it moves to the $x$-axis (so it is at the point $(1+Y,0)$ at that  moment), it continues to evolve  as follows. We pick independently the random slope $s\in \{-1,1\}$ (each choice has probability $1/2$) and then move the pair $(\mathfrak{f},\mathfrak{g})$ along the line of slope $s$, either to the point on the line $y=-s\gamma x$, or to the point $(1+Y+\delta,\delta s)$. If the pair visits the line $y=-s\gamma  x$, the evolution stops. Otherwise, the pair moves along the line of slope $-s$, either to the line $y=s\gamma x$ or to $(1+Y+2\delta,0)$. In the first case the evolution stops, while in the second, we pick a new random slope $s$, and the pattern is repeated.

Let us list several properties of $(\mathfrak{f},\mathfrak{g})$, which follow directly from the above definition. First, it is easy to see that $|d\mathfrak{g}_n|\equiv |d\mathfrak{f}_n|$ for each $n\geq 1$. Second, the above analysis clearly shows that $(\mathfrak{f},\mathfrak{g})$ converges almost surely to a random variable $(\mathfrak{f}_\infty,\mathfrak{g}_\infty)$ satisfying $|\mathfrak{g}_\infty|=\gamma \mathfrak{f}_\infty$,  almost surely. The final observation is that conditionally on the set $\{\mathfrak{g}_1=0\}$, the random variable $\mathfrak{g}_\infty$ is symmetric, while on $\{\mathfrak{g}_1>0\}$, the random variable  equals  $\gamma(1+Y)/(\gamma+1)$. Consequently, we get
\begin{align*}
 \E \mathfrak{g}_\infty|\mathfrak{g}_\infty|^{p-2}=\E \big\{\E \big[\mathfrak{g}_\infty|\mathfrak{g}_\infty|^{p-2}|\mathfrak{g}_1\big]\big\}&=\left(\frac{\gamma(1+Y)}{\gamma+1}\right)^{p-1}\mathbb{P}\left(\mathfrak{g}_1=\frac{\gamma(1+Y)}{\gamma+1}\right)\\
 &=Y\left(\frac{\gamma(1+Y)}{\gamma+1}\right)^{p-2}.
\end{align*}

In what follows, we will require the asymptotic behavior of the $p$-th moment of $\mathfrak{f}_\infty$ as $\delta\to 0$. It will be convenient to use the notation $A \simeq B$ when $\lim_{\delta\to 0}A/B=1$. Directly from (i)-(v), we derive  that $\mathbb{P}(\mathfrak{f}_\infty\geq (1+Y)/(\gamma+1))=1$ and, for $k\geq 1$,
$$ \mathbb{P}\left(\mathfrak{f}_\infty\geq \frac{1+Y}{\gamma+1}(1+2\delta)^{k}\right)=\frac{\gamma-Y}{(1+Y)(\gamma+\delta(\gamma+1))}\mathcal{P}^{k-1},$$
where
$$ \mathcal{P}=\frac{\gamma+\delta(\gamma-1)}{(1+2\delta)(\gamma+\delta(\gamma+1)}.$$
Substituting this gives 
$$ \mathbb{P}\left(\mathfrak{f}_\infty=\frac{1+Y}{\gamma+1}\right)\simeq \frac{(\gamma+1)Y}{(1+Y)\gamma}$$
and, for $k\geq 1$,
\begin{align*}
 \mathbb{P}\left(\mathfrak{f}_\infty=\frac{1+Y}{\gamma+1}(1+2\delta)^k\right)
=\frac{\gamma-Y}{(\gamma+\delta(\gamma+1))(1+Y)}\mathcal{P}^{k-1}(1-\mathcal{P}).
\end{align*}
Integrating these distributions gives that 
\begin{align*}
\E |\mathfrak{f}_\infty|^p&\simeq \frac{(\gamma+1)Y}{(1+Y)\gamma}\left(\frac{1+Y}{\gamma+1}\right)^p+\frac{\gamma-Y}{\gamma(1+Y)}\sum_{k=1}^\infty \left(\frac{1+Y}{\gamma+1}(1+2\delta)^k\right)^p\mathcal{P}^{k-1}(1-\mathcal{P})\\
&\simeq \frac{Y}{\gamma}\left(\frac{1+Y}{\gamma+1}\right)^{p-1}\\
&\quad +\left(\frac{1+Y}{\gamma+1}\right)^{p-1}\frac{\gamma-Y}{\gamma^2}\cdot 2\delta\sum_{k=1}^\infty \left[\frac{(1+2\delta)^{p-1}(\gamma+\delta(\gamma-1))}{\gamma+\delta(\gamma+1)}\right]^{k-1}\\
&\simeq \frac{Y}{\gamma}\left(\frac{1+Y}{\gamma+1}\right)^{p-1}+\left(\frac{1+Y}{\gamma+1}\right)^{p-1}\frac{\gamma-Y}{\gamma}\times\\
&\qquad \qquad \qquad \times\frac{ 2\delta}{\gamma(1-(1+2\delta)^{p-1})+\delta(\gamma+1)-\delta(\gamma-1)(1+2\delta)^{p-1}}\\
&\simeq \frac{Y}{\gamma}\left(\frac{1+Y}{\gamma+1}\right)^{p-1}+\frac{\gamma-Y}{\gamma(1-\gamma(p-1))}\left(\frac{1+Y}{\gamma+1}\right)^{p-1}\\
&=\left(\frac{1+Y}{\gamma+1}\right)^{p-1}\frac{1-(p-1)Y}{1-(p-1)\gamma}.
\end{align*}
Here in the third passage we have used the fact that $\gamma<(p-1)^{-1}$.  This   guarantees that the geometric series converges and that the martingale $\mathfrak{f}$ is bounded in $L^p$. 

Equipped with the above facts concerning $(\mathfrak{f},\mathfrak{g})$, we are ready to prove the estimate $\mathbb{B}_p\geq \mathcal{B}_p$. Pick $(\zeta,\eta,Z,H)\in \mathcal{D}$ with $Z>|\zeta|^p$, $H>|\eta|^q$ and assume first that $|\eta|^qZ<|\zeta|^pH$. Let us decrease $Z$ and $H$ a little: that is, choose $\bar{Z}\in (|\zeta|^p,Z)$ and $\bar{H}\in (|\eta|^q,H)$ for which the condition $|\eta|^q\bar{Z}<|\zeta|^p\bar{H}$ is still satisfied. Let $\gamma$, $Y$ be the numbers given by  the system \eqref{intr} (with the parameters $\zeta$, $\eta$, $\bar{Z}$ and $\bar{H}$). Put $f=\zeta \mathfrak{f}$, $g=|\zeta|\eta'\mathfrak{g}$ and let $h$ be the martingale adapted to the filtration of $f$ and $g$, with the terminal value $h_\infty$ given by
$$ h_\infty= \mathfrak{g}_\infty|\mathfrak{g}_\infty|^{p-2}\cdot \left(\frac{\bar{H}|\zeta|^p}{\bar{Z}\gamma^p}\right)^{1/q}\eta'.$$
Here $\eta'=\eta/|\eta|$ if $\eta\neq 0$, and $0'$ would be  an arbitrary vector of length $1$. Since $\mathfrak{g}_\infty$ belongs to $L^p$, the martingale $h$ is bounded in $L^q$. 
We have $\E f_\infty=\zeta\E \mathfrak{f}_\infty=\zeta$ and  furthermore, as $\delta$ approaches $0$, the $p$-th moment $\E |f_\infty|^p$ converges to $\bar{Z}$ (by the above calculation). Therefore we have $\E |f_\infty|^p\leq Z$,  for sufficiently small $\delta$. But 
$$ \E h=\left(\frac{\bar{H}|\zeta|^p}{\bar{Z}\gamma^p}\right)^{1/q}Y\left(\frac{\gamma(1+Y)}{1+\gamma}\right)^{p-2}\eta'=\frac{\bar{H}^{1/q}|\zeta|^{p/q}}{\bar{Z}^{1/q}}\frac{Y(1+Y)^{p-2}}{\gamma(1+\gamma)^{p-2}}\eta'=
|\eta|\eta'=\eta,$$
where in the third passage we again used  \eqref{intr}. Furthermore, we have
$$ \E |h|^q=\frac{\bar{H}|\zeta|^p}{\bar{Z}\gamma^p}\E |\mathfrak{g}_\infty|^p \xrightarrow{\delta\to 0} \frac{\bar{H}|\zeta|^p}{\bar{Z}}\cdot \left(\frac{1+Y}{\gamma+1}\right)^{p-1}\frac{1-(p-1)Y}{1-(p-1)\gamma}= \bar{H},$$
and hence $\E |h|^q\leq H$,  if $\delta$ is small enough. Therefore, by the very definition of $\mathbb{B}_p$,
\begin{align*}
 \mathbb{B}(\zeta,\eta,Z,H)&\geq \E \langle g_\infty,h_\infty\rangle-\langle \E g_\infty,\E h_\infty\rangle\\
&=\E |\mathfrak{g}_\infty|^p\cdot |\zeta|\left(\frac{\bar{H}|\zeta|^p}{\bar{Z}\gamma^p}\right)^{1/q}-|\eta\zeta|Y. \\
\end{align*}
However, as we have already observed above, $\E |\mathfrak{g}_\infty|^p$ converges to $\bar{Z}\gamma^p/|\zeta|^p$,  as $\delta\to 0$. This implies
 $$  \mathbb{B}_p(\zeta,\eta,Z,H)\geq \frac{\bar{Z}\gamma^p}{|\zeta|^p}|\zeta|\left(\frac{\bar{H}|\zeta|^p}{\bar{Z}\gamma^p}\right)^{1/q}-|\eta\zeta|Y=\mathcal{B}_p(\zeta,\eta,\bar{Z},\bar{H}).$$
Letting $\bar{Z}\to Z$, $\bar{H}\to H$ gives the desired inequality $ \mathbb{B}_p(\zeta,\eta,Z,H)\geq \mathcal{B}_p(\zeta,\eta,Z,H).$ 

Finally, we turn our attention to the case $|\eta|^qZ\geq |\zeta|^pH$. As before, we slightly decrease $Z$ and $H$,  picking $\bar{Z}\in (|\zeta|^p,Z)$, $\bar{H}\in (|\eta|^q,H)$ such that $|\eta|^q\bar{Z}\geq |\zeta|^p\bar{H}$. We will need  the following modification of the above martingale pair $(\mathfrak{f},\mathfrak{g})$. Let $Y$ be the number given by \eqref{specY} (with $Z$, $H$ replaced by $\bar{Z}$, $\bar{H}$), fix $\gamma<(p-1)^{-1}$ and take
$$ \e=(1-(p-1)\gamma)\cdot \frac{Y(1+\gamma)^{p-1}}{(1+Y)^p}\left(\frac{\bar{Z}}{|\zeta|^p}-1\right).$$
Let $(\mathfrak{f},\mathfrak{g})$ be a martingale satisfying three conditions. 
\begin{itemize}
\item[(i)] $(\mathfrak{f}_0,\mathfrak{g}_0)=(1,Y)$.
\item[(ii)] At the first step, the pair moves to $(1-\e,Y+\e)$ or to $(1+Y,0)$.
\item[(iii)] Starting with the second step, the pair moves according to the rules (ii)-(v) listed in the previous case.
\end{itemize}
As before we easily see that the condition $|d\mathfrak{g}_n|=|d\mathfrak{f}_n|$, $n\geq 1$, is satisfied. 
Now, put $f=\zeta \mathfrak{f}$, $g=|\zeta|\eta'\mathfrak{g}$ and let $h$ be the martingale with the terminal random variable
$$h_\infty=Y^{-1}(Y+\e)^{2-p}\mathfrak{g}_\infty|\mathfrak{g}_\infty|^{p-2}\eta.$$ We have $\E f_\infty=\zeta\E\mathfrak{f}_\infty=\zeta$ and, by the above definition of $\e$,
$$ \E |f_\infty|^p=\frac{Y}{Y+\e}|\zeta|^p(1-\e)^p+\frac{\e}{Y+\e}|\zeta|^p\frac{(1+Y)^p}{(1-(p-1)\gamma)(1+\gamma)^{p-1}} \to \bar{Z},$$
as $\delta\to 0$ and $\gamma\to (p-1)^{-1}$. Next, we check that 
$$ \E h_\infty=\frac{Y}{Y+\e}\cdot \frac{Y+\e}{Y}\eta=\eta.$$
By \eqref{specY} we also have 
\begin{align*}
 \E |h_\infty|^q=|\eta|^qY^{-q}(Y+\e)^{(2-p)q}\E |\mathfrak{g}_\infty|^p&\to \frac{|\eta|^p}{Y^p}\left[Y^p+\gamma^p\left(\frac{\bar{Z}}{|\zeta|^p}-1\right)\right]\\
 &=|\eta|^p+\frac{(p-1)^{-p}\left(\bar{Z}/|\zeta|^p-1\right)|\eta|^p}{Y^p}=\bar{H},
\end{align*}
as $\delta\to 0$ and $\gamma\to (p-1)^{-1}$. Consequently, since $\bar{Z}<Z$ and $\bar{H}<H$, we can write, for $\delta$ sufficiently small and $\gamma$ sufficiently close to $(p-1)^{-1}$, 
\begin{align*}
 \mathbb{B}(\zeta,\eta,Z,H)&\geq \E \langle g_\infty,h_\infty\rangle-\langle \E g_\infty,\E h_\infty\rangle \\
&=  |\zeta||\eta|\frac{(Y+\e)^{2-p}}{Y}\E |\mathfrak{g}_\infty|^p-|\zeta||\eta|Y.
\end{align*}
Letting $\delta\to 0$ and $\gamma \to (p-1)^{-1}$ (then $\e\to 0$), we see that the latter expression converges to
 \begin{align*}
\frac{|\zeta||\eta|}{Y^{p-1}}\left[Y^p+(p-1)^{-p}\left(\frac{\bar{Z}}{|\zeta|^p}-1\right)\right]-|\zeta||\eta|Y
 &=\frac{|\zeta||\eta|(p-1)^{-p}(\bar{Z}/|\zeta|^p-1)}{Y^{p-1}}.
\end{align*}
Plugging the formula \eqref{specY} for $Y$ we obtain $\mathbb{B}_p(\zeta,\eta,Z,H)\geq \mathcal{B}_p(\zeta,\eta,\bar{Z},\bar{H})$. It remains to let $\bar{Z}\to Z$ and $\bar{H}\to H$ to  complete the proof.

\end{document}